\documentclass{amsart}
\usepackage{amssymb, amsmath, latexsym, mathabx, dsfont,tikz}
\usepackage{multirow}

\renewcommand{\baselinestretch}{\baselinestretch}
\renewcommand{\baselinestretch}{1.1}
\numberwithin{equation}{section}

\newtheorem{thm}{Theorem}[section]
\newtheorem{lem}[thm]{Lemma}
\newtheorem{cor}[thm]{Corollary}
\newtheorem{prop}[thm]{Proposition}

\theoremstyle{definition}
\newtheorem{defn}[thm]{Definition}

\theoremstyle{remark}
\newtheorem{rmk}[thm]{Remark}

\numberwithin{equation}{section}

\newcommand{\ra}{{\rightarrow}}

\newcommand{\gen}{\text{gen}}

\newcommand{\ord}{\text{ord}}

\newcommand{\z}{{\mathbb Z}}
\newcommand{\q}{{\mathbb Q}}
\newcommand{\f}{{\mathbb F}}

\newcommand{\n}{{\mathbb N}}
\newcommand{\e}{{\epsilon}}

\begin{document}
\title[The number of representations of squares]{The number of representations of squares by integral ternary quadratic forms (II)}

\author{Kyoungmin Kim and Byeong-Kweon Oh}

\address{Department of Mathematical Sciences, Seoul National University, Seoul 151-747, Korea}
\email{kiny30@snu.ac.kr}

\address{Department of Mathematical Sciences and Research Institute of Mathematics, Seoul National University, Seoul 151-747, Korea}
\email{bkoh@snu.ac.kr}

\thanks{This work  was supported by the National Research Foundation of Korea (NRF-2014R1A1A2056296).}

\subjclass[2000]{Primary 11E12, 11E20} \keywords{Representations of ternary quadratic forms, squares}


\begin{abstract} Let $f$ be a positive definite ternary quadratic form. We assume that $f$ is non-classic integral, that is, the norm ideal of $f$ is $\z$.   
We say $f$ is {\it strongly $s$-regular  } if the number of representations of squares of integers by $f$ satisfies the condition in Cooper and Lam's conjecture in \cite {cl}. In this article, we prove that there are only finitely many strongly $s$-regular  ternary forms  up to equivalence if the minimum of the non zero squares that are represented by the form is fixed. In particular, we show that there are exactly $207$ non-classic integral strongly $s$-regular  ternary forms that represent one (see Tables 1 and 2). This result might be considered as a complete answer to a natural extension of Cooper and Lam's conjecture.       

\end{abstract}

\maketitle
\section{Introduction}
For a ternary quadratic form 
$$
f(x_1,x_2,x_3)=\sum_{1 \le i,j\le 3}a_{ij}x_ix_j, \ \ (a_{ij}=a_{ji})
$$ 
we say that $f$ is positive definite if the corresponding matrix $M_f:=(a_{ij})$ is positive definite. The discriminant $df$ of $f$ is defined by $\det(M_f)$. We say $f$ is non-classic integral if the ideal generated by $f(x_1,x_2,x_3)$ for any $(x_1,x_2,x_3) \in \z^3$ is $\z$. This is equivalent to 
$$
a_{ii}, a_{ij}+a_{ji} \in \z \ \ \text{for any}  \ \  i, j \in \{1,2,3\} \quad \text{and} \quad \gcd(a_{11},a_{22}, a_{33}, 2a_{23}, 2a_{31}, 2a_{12})=1.
$$ 
Throughout this article, we assume that every ternary quadratic form $f$ is positive definite and non-classic integral.

For a ternary quadratic form $f$ and an integer $n$, we define  
$$
R(n,f)=\{ (x_1,x_2,x_3) \in \z^3 : f(x_1,x_2,x_3)=n\} \qquad \text{and} \qquad r(n,f)=\vert R(n,f)\vert.
$$ 
Since we are assuming that $f$ is positive definite,   $R(n,f)$ is always  finite. If $r(n,f)$ is nonzero, then we say that $n$ is represented by $f$.  If
the class number of a ternary quadratic form $f$ is one, then for any integer $n$, $r(n,f)$ can be written as a product of local densities that are effectively computable.  
However, as far as the authors know, there is no known effective way to determine $r(n,f)$, for an arbitrary positive ternary quadratic form $f$.

Let $f$ be a  ternary quadratic form. 
 For any integer $n$, let $n_1$ and $n_2$ be positive integers such that 
  $P(n_1)\subset P(8df)$, $(n_2,8df)=1$ and $n=n_1n_2$. Here  $P(n)$ is denoted by the set of prime factors of $n$. 
  The  ternary quadratic form $f$ is called  {\it strongly $s$-regular}  if for any positive integer $n=n_1n_2$,
$$
r(n_1^2n_2^2,f)=r(n_1^2,f)\cdot \prod_{p\nmid 8df}h_p(df,\lambda_p),
$$
where $\lambda_p=\ord_p(n)$ for any prime $p$ and
$$ 
h_p(df,\lambda_p)=\frac{p^{\lambda_p +1}-1}{p-1}-\left(\frac{-df}{p}\right)\frac{p^{\lambda_p}-1}{p-1}.
$$  
  Clearly, if $f$ does not represent any squares of integers, then  $f$ is trivially strongly $s$-regular. So, throughout this article, we always assume that a strongly $s$-regular  ternary form $f$ represents at least one square of an integer. Note that this condition is equivalent to the condition that $f$ represents one over $\q$.

For a positive definite ternary form $f$ and an integer $n$, we define
$$
w(f)=\sum_{[g] \in \gen(f)} \frac1{o(g)} \qquad \text{and} \qquad r(n,\gen(f))=\frac1{w(f)}\sum_{[g] \in \gen(f)} \frac{r(n,g)}{o(g)},
$$ 
where $[g]$ is the equivalence class containing $g$ in the genus $\gen(f)$ of $f$ and $o(f)$ is the order of the isometry group $O(f)$. 
Then the Minkowski-Siegel formula  (see also Theorem 3.2 of \cite{ko}) says that 
$$
r(n_1^2n_2^2,\gen(f))=r(n_1^2,\gen(f))\prod_{p\nmid 8df} h_p(df,\lambda_p).
$$ 
 For any ternary form $f$ with class number one, since $r(n^2,f)=r(n^2,\gen(f))$  for any integer $n$, $f$ is strongly $s$-regular. 
 In the previous article \cite {ko}, we proved that every ternary quadratic form in the genus of $f$ is strongly $s$-regular   if and only if the genus of $f$ is indistinguishable by squares, that is $r(n^2,f')=r(n^2,f)$ for any $f' \in \gen(f)$ and any integer $n$. Also, we completely resolved the conjecture given by Cooper and Lam in \cite {cl}.
 
In this article, we prove that every strongly $s$-regular   form represents all squares that are represented by its genus, and  there are only finitely many strongly $s$-regular   ternary forms  up to equivalence if 
$$
m_s(f)=\min_{n \in \z^{+}} \{ n : r(n^2,f)\ne 0\}
$$ 
is fixed. Furthermore, we show that there are exactly  $207$ strongly $s$-regular   ternary quadratic forms  that represent one (see Tables 1 and 2). In the proof of Lemma \ref {important-1} and Theorem \ref {important-3}, we extensively use mathematics software MAPLE for large amount of computation. 

The term lattice will always refer to a positive definite $\z$-lattice on an $n$-dimensional positive definite quadratic space over $\q$. Let $L=\z x_1+\z x_2+ \dots+\z x_n$ be a $\z$-lattice of rank $n$. We write
$$
L\simeq (B(x_i,x_j)).
$$
The right hand side matrix is called a {\it matrix presentation} of
$L$. If $B(x_i,x_j)=0$ for any $i\ne j$, then we write $L \simeq \langle Q(x_1),Q(x_2),\dots,Q(x_n)\rangle$.

For a ternary $\z$-lattice $L=\z x_1+\z x_2+\z x_3$, the corresponding ternary quadratic form $f_L$ is defined by 
$$
f_L=\sum_{1\le i,j\le 3} B(x_i,x_j) x_ix_j.
$$
 We always assume that a $\z$-lattice $L$ is positive definite and non-classic integral. Recall that a ternary $\z$-lattice $L$ is non-classic integral if the norm ideal $\mathfrak n(L)$ of $L$ is $\z$. Hence $4\cdot dL$ is an integer.
 If $dL$ is not an integer, then the Legendre symbol $\left( \frac {dL} p\right)$ for an odd prime $p$ is defined as  $\left( \frac {4\cdot dL} p\right)$. For any odd integer $n$, we say $n$ (does not) divides $dL$ if $n$ (does not, respectively) divides the integer $4\cdot dL$.

 A binary form $ax^2+bxy+cy^2$ will be denoted by $[a,b,c]$ and a ternary form $ax^2+by^2+cz^2+dyz+ezx+fxy$ will be denoted by $[a,b,c,d,e,f]$. 
   
If an integer $n$ is represented by $L$ over $\z_p$ for any prime $p$ including infinite prime, then we say that $n$ is represented by the genus of $L$, and we write $n \,\,\ra \,\,\gen(L)$.  
When $n$ is represented by the lattice $L$ itself, then we write $n\,\,\ra\,\, L$.   
We always assume that $\Delta_p$ is a non square unit in $\z_p^{\times}$ for any odd prime $p$.

Any unexplained notations and terminologies can be found in  \cite {ki} or \cite {om}.

\section{Strongly $s$-regular ternary lattices}

Let $L$ be a strongly $s$-regular  ternary $\z$-lattice. Since we are assuming that  the genus of $L$ represents at least one square of an integer, we always have
$$
d(L\otimes \q_p)\ne-1 \quad \text{or} \quad S_p(L\otimes \q_p)=(-1,-1)_p \quad  \text{for any prime} ~~ p.
$$

\begin{lem} \label {notrepresent1} Any strongly $s$-regular  ternary $\z$-lattice $L$ represents all squares of integers that are represented by its genus. 
\end{lem}

\begin{proof}
Let $L$ be a strongly $s$-regular  ternary $\z$-lattice. Suppose that there is an integer $a$ such that  $a^2$ is represented by the genus of $L$, whereas it is not represented by $L$ itself.
Then for any prime $p\nmid 8dL$, if $a=p^t\cdot b$ for some integer $b$ such that $(b,p)=1$, then we have 
$$
r(p^2a^2,L)=r(b^2,L)\left(\frac{p^{t+2}-1}{p-1}-\left(\frac{-dL}{p}\right)\frac{p^{t+1}-1}{p-1} \right)=0.
$$
On the other hand, if we consider the action of the Hecke operator $T(p^2)$ to the theta series given by $L$ for any prime $p\nmid 8dL$,  then we have 
$$
r(p^2a^2,L)+\left(\frac{-dL}{p}\right)r(a^2,L)+p\cdot r\left(\frac{a^2}{p^2},L\right)=\sum_{[L']\in \gen(L)} \frac{r^*(pL',L)}{o(L')}r(a^2,L'),
$$
where $r^*(pL',L)$ is the number of primitive representations of $pL'$ by $L$.   Since $r(a^2,L)=r\left(\frac{a^2}{p^2},L\right)=0,$ we have 
$$
r(p^2a^2,L)=\sum_{[L']\in \gen(L)} \frac{r^*(pL',L)}{o(L')}r(a^2,L').
$$ 
From the assumption, there is a $\z$-lattice $L'\in \gen(L)$ such that $r(a^2,L')\ne0$.  Furthermore, by Class Linkage Theorem given by \cite {hjs}, there is a prime $q\nmid 8dL$ such that $r^*(qL',L)\ne0$. These imply that $r(q^2a^2,L)\ne 0$, which is a contradiction.
\end{proof}

\begin{cor} \label{localcom} Let $L$ be a strongly $s$-regular  ternary $\z$-lattice. Then every integer $m$ such that $m^2$ is represented by $L$ is a multiple of $m_s(L)=m_s(f_L)$. 
\end{cor}

\begin{proof} The corollary follows directly  from the fact that for any prime $p$, $\ord_p(m_s(L))$ is completely determined by $L_p$ by Lemma \ref{notrepresent1}.  
\end{proof}

Let $L$ be a ternary $\z$-lattice. For any prime $p$, the $\lambda_p$-transformation (or Watson transformation) is defined as follows: 
$$
\Lambda_p(L)= \{ x \in L : Q(x + z) \equiv Q(z) \  (\text{mod} \ p) \mbox{ for
all $z \in L$}\}.
$$
Let $\lambda_p(L)$ be the non-classic integral
lattice obtained from $\Lambda_p(L)$ by scaling $V=L\otimes \mathbb Q$ by a suitable
rational number. For a positive integer $N=p_1^{e_1}p_2^{e_2}\cdots p_k^{e_k}$,   we  also define 
$$
\lambda_N(L)=\prod_{i=1}^k \lambda_{p_i}^{e_i}(L).
$$
Note that $\lambda_p(\lambda_q(L))=\lambda_q(\lambda_p(L))$ for any primes $p \ne q$. 
\begin{lem} \label{aniso}
Let $L$ be a ternary $\z$-lattice and let $p$ be an odd prime.
If the unimodular component in a Jordan decomposition of $L_p$ is anisotropic, then
$$
r(pn,L)=r(pn,\Lambda_p(L)).
$$
\end{lem}

\begin{proof} See \cite {co}. \end{proof}

\begin{prop} \label {odd} Let $q$ be an odd prime and let $L$ be a ternary $\z$-lattice such that $L_q$ does not represent $1$.  Assume that $L_q\simeq\langle\Delta_q,q^{\alpha}\e_1,q^{\beta}\e_2\rangle$ for $\e_1,\e_2\in\z_q^{\times}$. 
\begin{itemize}
\item [(i)] If $\beta \ge \alpha \ge 2$, then $L$ is strongly $s$-regular  if and only if $\lambda_q(L)$ is strongly $s$-regular. Furthermore, if one of them is true, then   $m_s(L)=q\cdot m_s(\lambda_q(L))$.
\item [(ii)] If $\alpha=1$ and $\beta \ge 2$, then $L$ is  strongly $s$-regular    if and only if $\lambda_q^2(L)$ is strongly $s$-regular. Furthermore, if one of them is true, then   $m_s(L)=q\cdot m_s(\lambda_q^2(L))$.
\end{itemize}
\end{prop}
\begin{proof} Since  the proof is quite similar to each other, we only provide  the proof of the case when $\beta \ge \alpha \ge2$.  For any positive integer $n$, let $n_1$ and $n_2$ be positive integers such that 
  $P(n_1)\subset P(8dL)$, $(n_2,8dL)=1$ and $n=n_1n_2$, where $P(n)$ is the set of primes dividing $n$. Suppose that $L$ is strongly $s$-regular. Then we have
$$
r(q^2n_1^2n_2^2,L)=r(q^2n_1^2,L)\prod_{p\nmid8dL}h_p(dL,\lambda_p),
$$
where $\lambda_p$ and $h_p(dL,\lambda_p)$ are defined in the introduction.
Since $r(q^2n_1^2n_2^2,L)=r(n_1^2n_2^2,\lambda_q(L))$ and $r(q^2n_1^2,L)=r(n_1^2,\lambda_q(L))$ by Lemma \ref {aniso}, we have  
$$
r(n_1^2n_2^2,\lambda_q(L))=r(n_1^2,\lambda_q(L))\prod_{p\nmid8dL}h_p(L_p,\lambda_p),
$$
which implies that $\lambda_q(L)$ is strongly $s$-regular.

Conversely, Suppose that $\lambda_q(L)$ is strongly $s$-regular.  Then we have
$$
r(n_1^2n_2^2,\lambda_q(L))=r(n_1^2,\lambda_q(L))\prod_{p\nmid8dL}h_p(L_p,\lambda_p).
$$
Hence if $\ord_q(n_1)\ge1$, then 
$$
r(n_1^2n_2^2,L)=r(n_1^2,L)\prod_{p\nmid8dL}h_p(L_p,\lambda_p).
$$
Note that if $\ord_q(n_1)=0$, then $r(n_1^2n_2^2,L)=r(n_1^2,L)=0$.
Therefore $L$ is a strongly $s$-regular  lattice.  

Now assume that $L$ or $\lambda_q(L)$ is strongly $s$-regular. Since $1$ is not represented by $L_q$, $m_s(L)$ is divisible by $q$. Furthermore, since $r(q^2n,L)=r(n,\lambda_q(L))$ by Lemma \ref{aniso}, we have
$m_s(L)=q\cdot m_s(\lambda_q(L))$. \end{proof}


\begin{prop} \label {even}
Let $L$ be a ternary $\z$-lattice such that $L_2$ does not represent $1$.  Assume that  $L_2\simeq \langle\e_1\rangle\perp M$ for $\e_1\in\z_2^{\times}$. 
\begin{itemize}
\item [(i)]If $M$ is an improper modular lattice with norm contained in $4\z_2$ or  $M\simeq \langle 2^{\alpha}\e_2,2^{\beta}\e_3\rangle$ for $\e_2,\e_3\in\z_2^{\times}$ and nonnegative integers $\alpha, \beta$ such that  $\beta \ge \alpha\ge 2$, then $L$ is strongly $s$-regular  if and only if $\lambda_2(L)$ is strongly $s$-regular. 
Furthermore, if one of them is true, then   $m_s(L)=2\cdot m_s(\lambda_2(L))$.
\item [(ii)]  If $M\simeq \langle 2^{\alpha}\e_2,2^{\beta}\e_3\rangle$ with $\e_2,\e_3\in\z_2^{\times}$ and nonnegative integers $\alpha, \beta \  (\beta \ge \alpha)$ such that   $0 \le \alpha \le 1$, then $L$ is strongly $s$-regular  if and only if $\lambda_2^2(L)$ is strongly $s$-regular. Furthermore, if one of them is true, then   $m_s(L)=2\cdot m_s(\lambda_2^2(L))$.  
\end{itemize}
\end{prop}

\begin{proof}
Since the proof is quite similar to the odd case, the proof is left to the readers. \end{proof}

\begin{thm} \label{down} Let $L$ be a strongly $s$-regular  ternary $\z$-lattice. Then there is a positive integer $N$  such that 
\begin{enumerate} 
\item $\lambda_N(L)$ is a strongly $s$-regular  lattice such that $m_s(\lambda_N(L))$ is odd square free;
\item for any prime $p$ dividing $m_s(\lambda_N(L))$,  $\lambda_N(L)_p \simeq \langle \Delta_p,p,-p \rangle$.  
\end{enumerate}          
\end{thm}

\begin{proof}  By Propositions \ref{odd} and \ref{even}, if $p^2$ divides $m_s(L)$ for some prime $p$, then $\lambda_p(L)$ or $\lambda_p^2(L)$ is also strongly $s$-regular.   Hence by taking $\lambda_p$-transformations to $L$ repeatedly, if needed, we may find an integer $n$ such that $\lambda_n(L)$ is a strongly $s$-regular  lattice such that $m_s(\lambda_n(L))$ is odd square free. If $m_s(\lambda_n(L))$ is one, then there is nothing to prove.  
 Assume that $m_s(\lambda_n(L))=p_1p_2\cdots p_t$ where $p_i\ne p_j$ are primes.  Assume that  $p=p_i$ for some $i=1,2,\dots,t$. Then 
$1$ is not represented by $\lambda_n(L)_p$ by Lemma \ref{notrepresent1}. Hence  by Proposition \ref{odd},  either  $\lambda_p^{\iota}(\lambda_n(L))$ is a strongly $s$-regular  lattice such that 
$$
p\cdot m_s(\lambda_p^{\iota}(\lambda_n(L)))=m_s(\lambda_n(L)) \quad \text{and} \quad 1 \ra (\lambda_p^{\iota}(\lambda_n(L)))_p,
$$ 
where $\iota=1$ or $2$  depending on the structure of $(\lambda_n(L))_p$, or
\begin{equation} \label{terminal}
L_p \simeq   \langle \Delta_p,p,-p \rangle.
\end{equation}     
 If $n'$ is the product of primes satisfying the first condition, then $N=n\cdot n'$ satisfies all conditions given in the statement of the theorem. 
\end{proof}

\begin{defn} A strongly $s$-regular  ternary $\z$-lattice is called {\it terminal} if $m_s(L)$ is an odd square free integer, and for any prime $p$ dividing $m_s(L)$, $L_p$ satisfies the above condition \eqref{terminal}.
\end{defn}

Note that for any strongly $s$-regular  lattice $L$, there is an integer $N$ such that $\lambda_N(L)$ is a terminal strongly $s$-regular  lattice by Theorem \ref{down}. Therefore, to classify all strongly $s$-regular  lattices, in some sense, it suffices to find all terminal strongly $s$-regular  lattices. 

\begin{rmk} In fact, there are infinitely many terminal strongly $s$-regular  ternary $\z$-lattices. To show this,  let $q$ be a prime such that $q \equiv 5 \pmod 8$. We prove that the diagonal ternary lattice $L(q)=\langle 2,q,q\rangle$ is a terminal strongly $s$-regular  $\z$-lattice for any prime $q$ satisfying the above condition. 

If $n$ is not divisible by $q$, then $r(n^2,L(q))=0$. Furthermore 
$$
r(q^2n^2,L(q))=r(qn^2,\lambda_q(L(q))),
$$ 
where $\lambda_q(L(q))=\langle 1,1,2q\rangle$.  Let $\lambda_q(L(q))=\z x_1+\z x_2+\z x_3$ 
such that $(B(x_i,x_j))=\text{diag}(1,1,2q)$. Let $u$ be an integer such that $u^2\equiv -1 \pmod q$. Let $z=ax_1+bx_2+cx_3 \in \lambda_q(L(q))$ such that $Q(z)=qn^2$. 
Then, since   
$$
Q(z)=a^2+b^2+2qc^2 \equiv a^2-u^2b^2 \equiv (a-ub)(a+ub) \equiv 0 \pmod q,
$$
$z \in L(q,+):=\z(qx_1)+\z(ux_1+x_2)+\z x_3$ or $z \in L(q,-):=\z(qx_1)+\z(-ux_1+x_2)+\z x_3$. Note that $L(q,+)\cap L(q,-)=\z(qx_1)+\z(qx_2)+\z x_3$. Furthermore, 
$d(L(q,\pm))=2q^3$ and the scale of each lattice is $q\z$. Hence, we have
\begin{equation} \label{exam1}
r(q^2n^2,L(q))=r(qn^2,\lambda_q(L(q)))=2r(n^2,\langle 1,1,2\rangle)-r(n^2,L(q)).
\end{equation}
Therefore if we use an induction on $\ord_q(n)$,  the assertion follows directly from the fact that $\langle 1,1,2\rangle$ is strongly $s$-regular. Furthermore, since every $\z$-lattice in the genus of $L(q)$ satisfies the equation \eqref{exam1}, the genus of $L(q)$ is indistinguishable by squares.  
\end{rmk}

\begin{thm} For any positive integer $m$, 
 there are only finitely many strongly $s$-regular  ternary $\z$-lattices $L$ with $m_s(L)=m$. 
\end{thm}

\begin{proof}
Let $L$ be a strongly $s$-regular  ternary $\z$-lattice with $m_s(L)=m$. Since for any ternary lattice $K$ and any prime $p$, there are only finitely many lattices whose $\lambda_p$-transformation is isometric to $K$,  it suffices to show that there are only finitely many terminal strongly $s$-regular  lattice $L$ such that $m_s(L)=m$ under the assumption that $m=q_1q_2\cdots q_s$ is an odd square free integer. When $m=1$, then we let $s=0$. 

 Let $\{x_1,x_2,x_3\}$ be a Minkowski reduced basis for $L$ such that
 $$
(B(x_i,x_j))\simeq \begin{pmatrix} a&f&e\\f&b&d\\e&d&c \end{pmatrix} \qquad (0\le a\le b \le c \quad \text{and} \quad 2|f|\le a,~2|e|\le a,~2|d|\le b).
$$
Recall that $a,b,c,2d,2e,2f$ are relatively prime integers and $L=[a,b,c,2d,2e,2f]$.  
Let $p_t$ be the $t$-th smallest odd prime so that $p_1=3$, $p_2=5$ and so on.  Define 
$$
t'=\min \{t\in \n \mid 4m^6p_{t}^4<p_1\cdots p_{t-1} \}.
$$
 Note that such an integer always exists  by Bertrand-Chebyshev Theorem. Let $t''$ be the smallest integer such that $2p_{t''}>6\cdot3^{s+1}$ and $t_0=\max\{t', t''\}$. Finally, let $t_1$  be the integer such that $p_1p_2\cdots p_{t_1-1} \mid dL$, but  $p_{t_1} \nmid dL$. 

First, assume that $t_1\ge t_0$.  Then  
$$
4m^6p_{t_1}^4<p_1p_2\cdots p_{t_1-1}<4dL\le 4abc\le 4ac^2\le 4m^2c^2.
$$ 
Hence $m^2p_{t_1}^2<c$ and  we have 
$$
r(m^2p_{t_1}^2,L)=r\left(m^2p_{t_1}^2,\begin{pmatrix}a&f\\f&b\end{pmatrix}\right)\le 6\cdot3^{s+1}.
$$
However, since $p_{t_1} \nmid 8mdL$, we have
$$
r(m^2p_{t_1}^2,L)=r(m^2,L)\left(p_{t_1}+1-\left(\frac{-dL}{p_{t_1}}\right)\right)\ge 2p_{t_1} \ge 2p_{t''}>6\cdot3^{s+1}.
$$
This is a contradiction.

Finally, assume that $t_1<t_0$. Choose a positive integer $\lambda_0$ such that $p_{t_1}^{\lambda_0}>3^{s+1}(2\lambda_0+1)$. If $c>m^2p_{t_1}^{2\lambda_0}$, then
$$
r(m^2p_{t_1}^{2\lambda_0},L)=r\left(m^2p_{t_1}^{2\lambda_0},\begin{pmatrix}a&f\\f&b\end{pmatrix}\right)\le6\cdot3^s(2\lambda_0+1).
$$
This is a contradiction for
$$
r(m^2p_{t_1}^{2\lambda_0},L)=r(m^2,L)\left( \frac{p_{t_1}^{\lambda_0+1}-1}{p_{t_1}-1}-\left(\frac{-dL}{p_{t_1}}\right)\frac{p_{t_1}^{\lambda_0}-1}{p_{t_1}-1}\right)\ge2p_{t_1}^{\lambda_0}.
$$
 Therefore we have $c\le m^2p_{t_1}^{2\lambda_0}$, which implies that the discriminant of $L$ is bounded by a constant depending only on $m$. This completes the proof. 
\end{proof}

\section{Strongly $s$-regular  ternary lattices representing one}

The aim of this section is to find all strongly $s$-regular  ternary lattices $L$ with $m_s(L)=1$. 
Recall that we are assuming that the norm $\mathfrak n(L)$ of a $\z$-lattice $L$ is $\z$. Hence the scale $\mathfrak s(L)$ of $L$ is $\z$ or $\frac12\z$.

\begin{lem} \label {important-1}  Let $L$ be a strongly $s$-regular  ternary $\z$-lattice with $m_s(L)=1$. If $\mathfrak s(L)=\z$  {\rm ($\mathfrak s(L)=\frac12\z$)}, then $dL$ is not divisible by at least one prime in $\{3,5,7\}$ {\rm ($\{3,5,7,11\}$, respectively)}. 
\end{lem}

\begin{proof}  Let $L$ be a strongly $s$-regular  ternary $\z$-lattice with $m_s(L)=1$.  First, assume that $\mathfrak s(L)=\frac12\z$. Let $\{x_1,x_2,x_3\}$ be a Minkowski reduced basis for $L$ such that
$$
(B(x_i,x_j))\simeq \begin{pmatrix} 1&e&d\\e&a&c\\d&c&b \end{pmatrix} \quad (1\le a\le b  ~~~\text{and}~~~ 0 \le 2e \le 1,~-1\le 2d \le 1,~0 \le 2c \le a),
$$
where $a,b,2c,2d,2e$ are all integers and at least one of $2c,2d$ and $2e$ is odd. Let $p_t$ be the $t$-th smallest odd prime. Suppose, on the contrary,  that $p_1p_2\cdots p_t \mid dL$, whereas $p_{t+1} \nmid dL$ for some $t \ge 4$. 

First, assume that $t=4$. Since $3\cdot5\cdot7\cdot11 \mid dL$ and $13\nmid dL$ by assumption, we have 
\begin{equation} \label{13}
r(13^2,L)=r(1,L)\left(13+1-\left(\frac{-dL}{13}\right)\right) \ge 26.
\end{equation}
If $b \ge 13^2+1$, then $r(13^2,L)=r(13^2,[1,2e,a]) \le 18$. 
This is a contradiction and hence we have $1 \le a \le b \le 169$. For all possible finite cases, we may check by direct computations that there are no ternary $\z$-lattice satisfying the equation \eqref{13}.
The case when $t=5$ or $6$ can be dealt with similar manner to this.

Finally, assume that $t \ge 7$. Since $p_{t+1} \nmid dL$, we have
$$
r(p_{t+1}^2,L)=r(1,L)\left(p_{t+1}+1-\left(\frac{-dL}{p_{t+1}}\right)\right) \ge 46.
$$
If $t\ge 7,$  then $4p_{t+1}^4 < p_1\cdots p_t\le 4dL \le 4ab \le 4b^2$ by Bertrand-Chebyshev Theorem. Hence we have $p_{t+1}^2 < b$. Therefore we have 
$$
r(p_{t+1}^2,L)=r(p_{t+1}^2,[1,2e,a]) \le 18,
$$ 
for any positive integer $a$. This is a contradiction.

Since the proof of the case when $\mathfrak s(L)=\z$ is quite similar to the above, the proof are left to the readers. 
\end{proof}

\begin{thm} \label {important-3} There are exactly $207$ strongly $s$-regular  ternary $\z$-lattices $L$ with $m_s(L)=1$, which are listed in Tables 1 and 2. 
\end{thm}

\begin{proof}  Every strongly $s$-regular ternary lattices representing $1$ is listed in Tables 1 and 2.  
In Table 1, all ternary lattices except those with dagger mark and $\langle 1\rangle \perp [4,4,9], \langle 1\rangle \perp[4,4,25]$ are class number one. Hence they are strongly $s$-regular.   There are exactly $12$ ternary lattices in Table 1 whose class number is $2$. The strongly $s$-regularities of all these lattices with dagger mark were already proved in \cite {ko}. Finally, both $\langle 1\rangle \perp [4,4,9]$ and $\langle 1\rangle \perp[4,4,25]$ highlighted in boldface have class number $3$, and the strongly $s$-regularities of these two lattices  will be proved in Proposition \ref{main}.

There are exactly $30$ strongly $s$-regular lattices $L$ such that $\mathfrak s(L)=\frac12 \z$ and $h(L)=2$, which are listed in Table $2$. In fact,  the $\z$-lattice $S_i$ in Table 2 has class number two and the other lattice in the genus is $T_i$, for any $1 \le i \le 15$.   The strongly $s$-regularities of these lattices will be considered in Proposition \ref{classnumber2}. Those lattices highlighted in boldface in Table 2 has class number 3, and the proof of the strongly $s$-regularities of these lattices will be given in Proposition \ref{classnumber3-1}.      

Let $L$ be a strongly $s$-regular  ternary $\z$-lattice. First, assume that $\mathfrak s(L)=\z$.  
Then $L=\langle 1\rangle \perp \ell$, for some binary lattice $\ell$ such that
$$
\ell =[a,2b,c] =\begin{pmatrix} a&b\\b&c\end{pmatrix} \qquad (0\le 2b\le a\le c).
$$ 
From the above theorem, the discriminant of $L$, which is $ac-b^2$, is not divisible by at least one  prime in $\{3,5,7\}$. We will use the fact that
if $p \nmid 2dL$, then
$$
r(p^{2t},L)=r(1,L)\left( \frac{p^{t+1}-1}{p-1}-\left(\frac{-dL}{p}\right)\frac{p^{t}-1}{p-1}\right).
$$
Assume that $L \simeq \langle 1\rangle \perp [1,0,s]$ for some positive integer $s$. If $3\nmid s$, then   
$$
r(9,L)=r(1,L)\cdot \left(4-\left(\frac{-dL}3\right)\right) \ge 12.
$$ 
Hence $s=1,2,4,5$ or $8$. Assume that $s=3s_1$, for some integer $s_1$ such that $5 \nmid s_1$.   Since $r(25,L) >r(25,[ 1,0,1])=12$, 
we have $3s_1 \le 25$.  Therefore $s_1=1,2,3,4,7$ or $8$. Assume that $s=15s_2$ for some integer $s_2$ such that $7 \nmid s_2$.  One may apply similar argument
to show that there does not exist a strongly $s$-regular  lattice in this case.   
 
 From now on, we assume that $r(1,L)=2$, that is, $a\ge 2$. 
 Assume that  $3\nmid dL$. 
Then we have $r(3^2,L)=6$ or $10$, and $r(3^4,L)=18$ or $34$. Hence we have $2\le a \le 9$. 
If $a=9$, then $c=9$.  In this case, one may easily show that 
$$
r\left(3^4,\langle1\rangle \perp \begin{pmatrix} 9&b\\b&9\end{pmatrix}\right) \ne 18, 34,
$$  
which is a contradiction. 
Next assume that $a=8$.  If $c\ge 10$, then 
$$
r(3^2,L)=r(3^2,[1,0,8])=6=2\left(4-\left(\frac{-dL}{3}\right)\right).
$$ 
Hence $\left(\frac{-dL}{3}\right)=1$, which  implies that $r(3^4,L)=18$.  If $c\ge82,$ then $r(3^4,L)=r(3^4,[1,0,8])=10$,  which is a contradiction. Therefore we have $8\le c \le 81$. 
For all possible finite cases, one may easily check only when $\ell$ is isometric to one of 
$$
\begin{array} {ll}
& [8,0,8], \ [8,0,10]^{\dagger}, \ [8,0,13]^{\dagger}, \ [8,0,16], \ [8,0,40], \\
&[8,4,18]^{\dagger}, \ [8,8,12], \ [8,8,24] \quad \text{and} \quad  [8,8,72],
\end{array}
$$ 
$L=\langle 1\rangle \perp \ell$ is  strongly $s$-regular.    Note that the class number of $L$ is one if $\ell$ is isometric to one of binary lattices given above, except binary lattices with dagger mark.  
When $\ell$ is isometric to one of binary lattices with dagger mark, the proof of the strongly $s$-regularity of $L$ is proved in \cite{ko}. 
The proof of the remaining cases, that is $2\le a\le 7$, is quite similar to this. In particular, the case when $\ell=[4,4,9]$, where the class number of $L=\langle 1\rangle \perp \ell$ is $3$ in this case, will be considered in Proposition \ref{main}.

Assume that $dL$ is divisible by $3$, but is not divisible by $5$.  In this case, we have $r(5^2,L)=10$ or $14$, and $r(5^4,L)=50$ or $74$.  Since $r(p^2, [1,0,a]) \le 6$ for any prime $p$ and any integer $a\ge 2$, we have $2 \le a \le c \le 25$.
For all possible cases, one may easily show that $L$ is strongly $s$-regular if and only if the class number of $L$ is one, except the case when $\ell=[4,4,25]$. For the exceptional case, the proof of the strongly $s$-regularity of $L$ will be given in Proposition \ref{main}. 

 Finally, assume that $15 \mid dL$ and $7\nmid dL$. For all possible cases, $L$ is strongly $s$-regular if and only if the class number of $L$ is one.

Now assume that $\mathfrak s(L)=\frac12 \z$. Let $\{x_1,x_2,x_3\}$ be a Minkowski reduced basis for $L$ such that
$$
(B(x_i,x_j))\simeq \begin{pmatrix} 1&e&d\\e&a&c\\d&c&b \end{pmatrix}=[1,a,b,2c,2d,2e],
$$
where $a,b,2c,2d,2e$ are integers such that $1\le a\le b$ and $0 \le 2e \le 1, \ -1\le 2d \le 1$, $0 \le 2c \le a$. Note that at least one of $2c,2d,2e$ is odd.   
In this case, the discriminant of $L$ is not divisible by at least one prime in $\{3,5,7,11\}$ by the above theorem. 
Assume that $a=1$ and $b=1$. Then clearly, $L\simeq [1,1,1,0,0,1]$  or $[1,1,1,1,1,1]$, all of which are strongly $s$-regular.   
 Next, assume that $a=1, e=0$ and $b \ge 2$. Let $p \in \{3,5,7,11\}$ be a prime not dividing $dL$. Since $r(p^2,L)=4p$ or $4(p+2)$ and $r(5^2,[1,0,1])=12$, $r(p^2,[1,0,1])=4$ for any $p \in \{3,7,11\}$,  we have $2 \le b \le p^2$.  One may easily show that there are exactly $6$ strongly $s$-regular  lattices in this case, all of which have class number $1$.   
 
  Next assume that $a=1, 2e=1$ and $b \ge 2$. In this case, since $r(1,L)=6$, we have $r(p^2,L)=6p$ or $6(p+2)$. Furthermore, since  $r(7^2,[1,1,1])=18$ and $r(p^2,[1,1,1])=6$ for any $p \in \{3,5,11\}$, we have $2\le b \le p^2$.  One may easily show that there are exactly $12$ strongly $s$-regular  lattices in this case, all of which have class number $1$.

From now on, we assume that $r(1,L)=2$, that is, $a \ge 2$. Assume further that $3 \nmid dL$. Since $r(3^2,L)=6$ or $10$, we have $2 \le a \le 9$. Furthermore, since  $r(3^4,L)=18$ or $34$, and  $r(3^{2n},[1,2e,a]) \le 2(2n+1)$ for any positive integer $n$,  we have 
$$
\begin{cases}
2 \le a \le b \le 9  \quad &\text{if $r(9,[1,2e,a])<6$}, \\ 
2 \le a \le b \le 81 \quad    &\text{if $r(9,[1,2e,a])=6.$}\\
 \end{cases}
 $$
 In this case, we have $30$ candidates of strongly $s$-regular  lattices as in the first row of Table $2$. Among them, there are exactly $18$ lattices having class number $1$.  Remaining $12$ lattices $T_1 \sim T_6$ and $S_1 \sim S_6$ in the first row of Table 2 has class number $2$. The proof of the strongly $s$-regularities of these lattices  will be considered in Proposition \ref {classnumber2}. 

Next assume that $3\mid dL$ and $5\nmid dL$. Since $r(5^2,L)=10$ or $14$, and 
$$
r(5^2,[1,2e,a])\le 6<10,
$$
 we have $2 \le a \le b \le 25$. In this case, we have twenty two candidates with class number 1, twelve lattices with class number 2, and three lattices with class number $3$. The proof of the strongly $s$-regularities of these lattices having class number $2$ (class number $3$) will be considered in Proposition \ref{classnumber2} (Proposition \ref{classnumber3-1}, respectively). Recall that all lattices highlighted in boldface in Tables 1 and 2 have class number $3$.

Now assume that $dL$ is divisible by $15$, but not divisible by $7$. In this case, we have $2\le a \le b \le 49$. Everything is quite similar to the above cases.  In this case, we have twelve candidates with class number 1, two lattices with class number 2, and one lattice with class number $3$. 

Finally, assume that $105\mid dL$ and $11\nmid dL$. In this case,  $L$ is isometric to one of
$4$ lattices listed in fourth line of Table 2. The proof of the strongly $s$-regularities of these 4 ternary lattices will be considered in Proposition \ref {classnumber2}. 
\end{proof}


\begin{table}[t]
\begin{tabular}{|c|l|}\hline
 &\hskip 4.5cm $\ell$ \\\hline
\multirow{8}{*}{$3\nmid dL$} & $[1,0,1]$, $[1,0,2]$, $[1,0,4]$, $[1,0,5]$, $[1,0,8]$, $[2,0,2]$, \\
 &  $[2,2,3]$, $[2,0,4]$, $[2,0,5]$, $[2,2,6]$, $[2,0,8]$, $[2,0,10]$,  \\
 &  $[2,0,13]^{\dagger}$, $[2,0,16]$, $[2,2,18]$, $[2,0,22]^{\dagger}$, $[2,2,33]^{\dagger}$, $[2,0,40]^{\dagger}$,\\
 &  $[2,0,70]^{\dagger}$, $[3,2,3]$, $[3,2,5]$, $[3,2,7]$, $[4,0,4]$, $[4,4,5]$,\\           
  &  $[4,0,8]$, $[4,4,8]$, $\textbf{[4,4,9]}$, $[5,0,5]$, $[5,4,6]^{\dagger}$, $[5,0,8]$,\\
 (47)&  $[5,0,10]$, $[5,4,12]$, $[5,0,13]^{\dagger}$, $[5,2,21]^{\dagger}$, $[5,0,25]$, $[5,0,40]$,\\
 &  $[6,4,6]$, $[6,4,8]^{\dagger}$, $[8,0,8]$, $[8,0,10]^{\dagger}$, $[8,8,12]$, $[8,0,13]^{\dagger}$,\\
 &  $[8,0,16]$, $[8,4,18]^{\dagger}$, $[8,8,24]$, $[8,0,40]$, $[8,8,72]$ \\\hline
   
\multirow{8}{*}{$3\mid dL$, $5\nmid dL$} & $[1,0,3]$, $[1,0,6]$, $[1,0,9]$, $[1,0,12]$, $[1,0,21]$, $[1,0,24]$, \\
 &  $[2,2,2]$, $[2,0,3]$, $[2,2,5]$, $[2,0,6]$, $[3,0,3]$, $[3,0,4]$,  \\
 &  $[3,0,6]$, $[3,0,9]$, $[3,0,12]$, $[3,0,18]$, $[4,4,4]$, $[4,0,6]$,\\
 &  $[4,4,7]$, $[4,0,12]$, $[4,4,13]$, $[4,0,24]$, $\textbf{[4,4,25]}$, $[5,2,5]$,\\           
 &  $[6,0,6]$, $[6,6,6]$, $[6,0,9]$, $[6,0,16]$, $[6,0,18]$, $[6,6,21]$,\\
(45) &  $[6,0,24]$, $[8,8,8]$, $[9,0,9]$, $[9,6,9]$, $[9,0,12]$, $[9,0,21]$,\\
 &  $[9,0,24]$, $[10,4,10]$,$[12,0,12]$, $[12,12,21]$, $[16,16,16]$,\\
 &  $[16,0,24]$, $[21,0,21]$, $[24,0,24]$, $[24,24,24]$ \\\hline
 
\multirow{2}{*}{$15\mid dL$, $7\nmid dL$} &  $[3,0,10]$, $[3,0,30]$, $[4,4,16]$, $[6,6,9]$, $[10,10,10]$, $[10,0,30]$, \\
 &  $[12,12,13]$, $[12,12,33]$, $[40,40,40] $   \\
 (9)& \\
 \hline                                         
\end{tabular}
\vskip 0.2cm
\caption{Strongly  $s$-regular lattices $L=\langle 1\rangle \perp \ell$}
\end{table}



\begin{table}[t]
\begin{tabular}{|c|l|} \hline
& \hskip 4.5cm $L$ \\\hline
\multirow{12}{*}{$3\nmid dL$}  
 & $[1,1,1,1,1,1]$, $[1,1,2,0,1,0]$, $[1,1,2,1,1,1]$,  \\
 & $[1,1,3,1,1,0]$, $[1,1,3,1,1,1]$, $[1,1,5,1,1,1]$,   \\
 & $[1,1,7,1,1,1]$, $[1,2,2,1,1,0]$, $[1,2,2,2,1,1]$,  \\
 & $[1,2,3,0,1,0]$, $[1,2,3,1,0,1]$, $T_1\!=\![1,2,4,1,1,1]$,  \\
 & $S_1\!=\![1,2,4,2,1,0]$, $[1,2,7,0,0,1]$, $[1,2,9,0,1,0]$,  \\
 & $S_2\!=\![1,2,23,0,1,0]$, $[1,3,3,2,1,1]$, $[1,3,4,2,0,1]$,  \\
 & $[1,3,5,1,1,1]$, $[1,3,5,3,1,1]$, $S_3\!=\![1,3,9,2,1,1]$,  \\
(37) & $S_4\!=\![1,3,10,0,0,1]$, $T_2\!=\![1,3,17,2,1,1]$, $[1,3,22,0,0,1]$,  \\
 & $[1,4,4,3,1,1]$, $[1,4,9,3,1,1]$, $T_3\!=\![1,5,5,1,1,0]$,  \\
 & $T_4\!=\![1,5,6,2,0,1]$, $S_5\!=\![1,5,19,5,1,0]$, $S_6\!=\![1,5,49,5,1,0]$,   \\
 & $[1,7,9,7,1,0]$, $[1,9,9,8,1,1]$, $T_5\!=\![1,9,10,0,0,1]$,  \\
 & $[1,9,15,5,0,1]$, $[1,9,21,7,0,1]$, $T_6\!=\![1,9,29,8,1,1]$, \\
 & $[1,9,70,0,0,1]$ \\
 \hline
   
\multirow{16}{*}{$3\mid dL$, $5\nmid dL$} 
 & $[1,1,1,0,0,1]$, $[1,1,2,0,0,1]$, $[1,1,2,1,1,0]$,   \\
 & $[1,1,3,0,0,1]$, $[1,1,4,0,0,1]$, $[1,1,5,1,1,0]$,   \\ 
 & $[1,1,6,0,0,1]$, $[1,1,11,1,1,0]$, $[1,1,12,0,0,1]$,  \\
 & $[1,1,18,0,0,1]$, $[1,2,2,1,1,1]$, $[1,2,3,1,1,0]$,  \\
 & $[1,2,3,2,1,0]$, $[1,2,4,2,1,1]$, $[1,2,5,1,1,1]$,  \\
 & $S_7\!=\![1,2,7,0,1,0]$, $S_8\!=\![1,2,9,2,1,0]$, $S_9\!=\![1,2,10,1,0,1]$,  \\ 
 & $[1,3,4,3,1,0]$, $T_7\!=\![1,3,5,1,0,1]$, $T_8\!=\![1,3,6,0,0,1]$,  \\
 & $L_1\!=\!\textbf{[1,3,7,0,1,0]}$, $[1,3,8,2,0,1]$, $[1,4,4,2,1,1]$,  \\
 & $[1,4,5,2,1,0]$, $T_9\!=\![1,4,5,2,1,1]$, $[1,4,6,3,0,1]$,  \\
(47) & $[1,4,11,2,1,0]$, $[1,4,13,2,1,1]$, $[1,5,5,4,1,1]$,  \\
 & $[1,5,7,1,0,1]$, $[1,5,7,2,1,1]$, $S_{10}\!=\![1,5,13,5,1,1]$,  \\
 & $S_{11}\!=\![1,5,15,3,0,1]$, $[1,6,7,0,1,0]$, $T_{10}\!=\![1,6,11,6,1,0]$,  \\
 & $S_{12}\!=\![1,6,25,0,1,0]$, $[1,7,7,5,1,1]$, $T_{11}\!=\![1,7,11,5,1,0]$,  \\
 & $L_4\!=\!\textbf{[1,7,12,0,0,1]}$, $[1,7,13,5,1,1]$, $[1,7,18,0,0,1]$,  \\
 & $M_6\!=\!\textbf{[1,7,19,5,1,1]}$, $[1,9,13,9,1,0]$, $T_{12}\!=\![1,13,13,8,1,1]$,  \\
 & $[1,13,15,3,0,1]$, $[1,13,23,13,1,0]$  \\
 \hline
 
\multirow{6}{*}{$15\mid dL$, $7\nmid dL$}
 & $[1,1,4,0,1,0]$  $[1,1,10,0,0,1]$, $[1,1,30,0,0,1]$,\\
 & $[1,2,7,2,1,1]$, $[1,3,3,1,1,1]$, $[1,4,5,4,1,0]$, \\
 & $[1.4.15,0,0,1]$, $[1,5,9,5,1,0]$, $[1,6,13,6,1,0]$, \\
 & $[1,7,7,3,1,0]$, $S_{13}\!=\![1,7,10,0,0,1]$, $T_{13}\!=\![1,7,11,5,1,1]$, \\
(18) & $L_{10}\!=\!\textbf{[1,7,30,0,0,1]}$, $[1,7,31,5,1,1]$, $[1,10,19,0,1,0]$, \\   
 & $[1,15,19,15,1,0]$, $[1,19,19,8,1,1]$, $[1,19,30,0,0,1]$ \\
 \hline                                         

\multirow{1}{*}{$105\mid dL$, $11\nmid dL$} &  $S_{14}\!=\![1,2,15,0,0,1]$, $T_{14}\!=\![1,4,7,0,0,1]$, \\       
(4) & $S_{15}\!=\![1,7,17,7,1,0]$, $T_{15}\!=\![1,11,11,7,1,1]$  \\
 \hline                                         
\end{tabular}
\vskip 0.2cm \caption{Strongly $s$-regular lattices $L$ with $\mathfrak s(L)=\frac12\z$}
\end{table}

\section{Non trivial strongly $s$-regular  ternary lattices}

In this section, we prove the strongly $s$-regularities of ternary lattices with class number greater than $1$ in Tables 1 and 2.

\begin{lem}\label{hecke-act}
Let $L$ be a ternary $\z$-lattice. For any $L' \in \gen(L)$, if $r(n^2,L)=r(n^2,L')$ for any integer $n$ whose prime factor divides $8dL$, then the genus of $L$ is indistinguishable by squares.   
\end{lem}
\begin{proof}
See Lemma 2.5 of \cite {ko}.
\end{proof}

Let $L$ be a ternary $\z$-lattice. Assume that the $\frac12\z_p$-modular component in a Jordan decomposition of $L_p$ is nonzero isotropic. Assume that $p$ is a prime dividing $4dL$.
Then by Weak Approximation Theorem, there exists a basis $\{x_1, x_2, x_3 \}$ for $L$ such that
$$
(B(x_i,x_j))\equiv\begin{pmatrix}0&\frac12\\ \frac12&0\end{pmatrix}\perp \langle p^{\ord_p(4dL)} \delta\rangle \ (\text{mod} \ p^{\ord_p(4dL)+1}),
$$
where $\delta$ is an integer not divisible by $p$. We define
$$
\Gamma_{p,1}(L) = \z px_1 + \z x_2+ \z x_3 \quad \text{and} \quad \Gamma_{p,2}(L) = \z x_1 + \z px_2+ \z x_3.
$$  
Note that the lattice $\Gamma_{p,i}(L)$ depends on the choice of basis for $L$. However the set
$\{\Gamma_{p,1}(L),\Gamma_{p,2}(L)\}$ is independent of the choices of the basis for $L$. 
 There are exactly two sublattices of $L$ with index $p$ whose norm is contained in $p\z$. They are, in fact, $\Gamma_{p,1}(L)$ and $\Gamma_{p,2}(L)$.  For some properties of these sublattices of $L$, see \cite {jlo}. 

\begin{lem} \label{iso}
Under the same assumptions given above,  we have 
$$
r(pn,L)=r(pn,\Gamma_{p,1}(L))+r(pn,\Gamma_{p,2}(L))-r(pn,\Lambda_p(L)).
$$ 
\end{lem}

\begin{proof}
See Proposition 4.1 of \cite{jlo}.
\end{proof}

\begin{prop} \label {classnumber2}
For any $i \ \ (1 \le i\le 15)$, the genus $\gen(S_i)$ is indistinguishable by squares. Therefore $S_i$ and $T_i$ are strongly $s$-regular  for any $1 \le i\le 15$.
\end{prop}
\begin{proof}

Since proofs are quite similar to each other, we only provide the proofs of the cases when $i=1,3,13,14$ and $15$ as representatives.  We put
$$
\begin{array}{llll}
&P_1=[2,4,4,2,2,0],      &P_2=[47,47,47,0,47,47],  &P_3=[1,1,10,0,0,1],\\
&Q=[4,8,16,2,4,4],       &S_{14,1}=[1,2,60,0,0,1], &S_{14,2}=[2,4,60,0,0,2],\\  
&S_{14,3}=[2,4,15,0,0,2],&T_{14,1}=[1,4,28,0,0,1], &T_{14,2}=[4,4,28,0,0,2],\\
&T_{14,3}=[4,4,7,0,0,2].& &\\
\end{array}
$$ 
   First, we consider the case when $i=1$. By Lemma \ref {aniso}, we see that $r(13^2n^2,S_1)=r(n^2,S_1)$ and $r(13^2n^2,T_1)=r(n^2,T_1)$ for any integer $n$. Also by Lemma \ref {iso}, we have
$$
r(4n^2,S_1)=2r(4n^2,P_1)-r(n^2,S_1)~\text{and}~ r(4n^2,T_1)=2r(4n^2,P_1)-r(n^2,T_1),
$$
for any integer $n$. Since $r(1,S_1)=r(1,T_1)$,  $\gen(S_1)$ is indistinguishable by squares by Lemma \ref {hecke-act}. 

We consider the case when $i=3$. Note that $d(S_3)=2^{-1}\cdot47$. By Lemma \ref {iso}, we have
$$
r(47^2n^2,S_3)=2r(47^2n^2,P_2)-r(n^2,S_3), r(47^2n^2,T_3)=2r(47^2n^2,P_2)-r(n^2,T_3),
$$
for any integer $n$. If $x^2+3y^2+9z^2+xy+2yz+zx=4n^2$, then $x,y,z$ are all even. Hence we have
$$
r(4n^2,S_3)=r(n^2,S_3)\quad\text{and}\quad
r(4n^2,T_3)=r(n^2,T_3).
$$
Therefore  $\gen(S_3)$ is indistinguishable by squares by Lemma \ref {hecke-act}.

Note that $d(S_{13})=2^{-1}\cdot 3^3\cdot 5$.  By Lemma \ref {aniso}, we have $r(3^2n^2,S_{13})=r(n^2,P_3)$, $r(5^2n^2,S_{13})=r(n^2,S_{13})$  and $r(3^2n^2,T_{13})=r(n^2,P_3)$, $r(5^2n^2,T_{13})=r(n^2,T_{13})$ for any integer $n$. If $x^2+7y^2+10z^2+xy=4n^2$, then $x,y,z$ are all even. Hence we have
$$
r(4n^2,S_{13})=r(n^2,S_{13})\quad\text{and}\quad
r(4n^2,T_{13})=r(n^2,T_{13}).
$$
Therefore the genus $\gen(S_{13})$ is indistinguishable by squares by Lemma \ref {hecke-act}.

Now, we consider the $\z$-lattice $S_{14}$, which is one of the most difficult cases.  Note that $d(S_{14})={2^{-2}}\cdot3\cdot5\cdot7$. By Lemma \ref {aniso}, we have
$$
r(p^2n^2,S_{14})=r(n^2,S_{14}) \quad \text{and} \quad r(p^2n^2,T_{14})=r(n^2,T_{14}), 
$$
for any prime $p\in\{3,5,7\}$. Let $\{x_1,x_2,x_3\}$ be the basis for $S_{14}$ such that $Q(ax_1+bx_2+cx_3)=a^2+ab+2b^2+15c^2$.  Assume that $Q(ax_1+bx_2+cx_3)=4n^2$.
Then we have $a\equiv c\pmod 2$, $b\equiv0\pmod 2$ or $c\equiv0 \pmod 2$. This implies that for $z=ax_1+bx_2+cx_3$, 
$$
z \in \z(2x_1)+\z(2x_2)+\z(x_1+x_3) \quad \text{or} \quad z \in \z(x_1)+\z(x_2)+\z(2x_3).
$$
Therefore we have, for any integer $n$, 
$$
r(4n^2,S_{14})=r(4n^2,Q)+r(4n^2,S_{14,1})-r(n^2,S_{14}).
$$
Similarly, we also have
$$
r(4n^2,T_{14})=r(4n^2,Q)+r(4n^2,T_{14,1})-r(n^2,T_{14}).
$$
Furthermore, one may easily show that
$$
\begin{array}{rl}
r(4n^2,S_{14,1})=&\!\!\!2r(4n^2,S_{14,2})-r(n^2,S_{14}),\\
r(4n^2,T_{14,1})=&\!\!\!2r(4n^2,T_{14,2})-r(n^2,T_{14}),
\end{array}
$$
and
\begin{equation} \label{14-1}
\begin{array}{rl}
r(4n^2,S_{14,3})=r(4n^2,S_{14,2})=&\!\!\!2r(n^2,S_{14})-r(n^2,S_{14,3}),\\
r(4n^2,T_{14,3})=r(4n^2,T_{14,2})=&\!\!\!2r(n^2,T_{14})-r(n^2,T_{14,3}).
\end{array}
\end{equation}
By combining all equalities given above, we have
\begin{equation} \label{14-2}
\begin{array}{rl} 
r(4n^2,S_{14})=&\!\!\!r(4n^2,Q)+2r(n^2,S_{14})-2r(n^2,S_{14,3}),\\
r(4n^2,T_{14})=&\!\!\!r(4n^2,Q)+2r(n^2,T_{14})-2r(n^2,T_{14,3}).
\end{array}
\end{equation}
Since $r(1,S_{14})=r(1,T_{14})=2$ and $r(1,S_{14,3})=r(1,T_{14,3})=0$, we have $r(2^{2t},S_{14})=r(2^{2t},T_{14})$ for any positive integer $t$ by \eqref{14-1} and \eqref{14-2}. Therefore the genus of $S_{14}$ is indistinguishable by squares by Lemma \ref {hecke-act}, and $r(n^2,S_{14,3})=r(n^2,T_{14,3})$ for any integer $n$. Note that the class number of $S_{14,3}$ is 3. In fact, the proof of the strongly $s$-regularities of $S_{9}$ is quite similar to this. 

Finally by Lemma \ref {aniso}, we have
$$
r(p^2n^2,S_{15})=r(n^2,S_{15}) \quad \text{and} \quad r(p^2n^2,T_{15})=r(n^2,T_{15}) 
$$ 
for any prime $p\in\{3,5,7\}$. Let $\{y_1,y_2,y_3\}$ be the basis for $S_{15}$ such that $Q(ay_1+by_2+cy_3)=a^2+7b^2+17c^2+7bc+ca$. Assume that $Q(ay_1+by_2+cy_3)=4n^2$. Then we have $a\equiv b\pmod 2$ and $c\equiv0\pmod 2$. Therefore we have 
$$
r(4n^2,S_{15})=r(4n^2,\z(2y_1)+\z(y_1+y_2)+\z(2y_3)),
$$
which implies that for any integer $n$, 
$$
r(4n^2,S_{15})=r(n^2,S_{14}).
$$
Similarly, we also have
$$
r(4n^2,T_{15})=r(n^2,T_{14}).
$$
Therefore  $\gen(S_{15})$ is indistinguishable by squares by Lemma \ref {hecke-act}.
\end{proof}

Now we consider the lattices having class number $3$.  We define
$$
K_{1,t}=\langle1\rangle \perp \begin{pmatrix} 4&2 \\2&2^33^t+1\end{pmatrix}, \  \ K_{2,t}=\langle1,1,2^53^t\rangle, \ \ K_{3,t}=\begin{pmatrix} 2&0&1 \\0&2&1 \\1&1&2^33^t+1\end{pmatrix},
$$
for any non negative integer $t$. 

\begin{lem} \label {classnumber3} For any non negative integer $t$, 
ternary $\z$-lattices $K_{1,t}, K_{2,t}$ and $K_{3,t}$ are in the same genus. Furthermore, we have
$$
2r(n^2,K_{1,t})=r(n^2,K_{2,t})+r(n^2,K_{3,t}),
$$
 for any integer $n$.
\end{lem}
\begin{proof} Note that $d(K_{i,t})=2^5\cdot 3^t$ for any $i=1,2,3$. 
By checking  local structures at $p=2$ and $3$, one may easily show that all ternary $\z$-lattices $K_{1,t}, K_{2,t}$ and $K_{3,t}$ are in the same genus for any integer $t\ge0$.
Fix a non negative integer $t$. 

Note that for any integer $n$, one may easily show that
$$
r(4n^2,K_{1,t})=r(4n^2,K_{2,t})=r(4n^2,K_{3,t})=r(n^2,\langle1,1,2^33^t\rangle).
$$  
Assume  that  $n$ is odd.   If $x^2+4y^2+4yz+(2^33^t+1)z^2=n^2$, then either $x$ or $z$ is odd, but not both. Hence we have 
 $$
 r(n^2,K_{1,t})=r(n^2,\langle1,4,2^53^t\rangle)+r\left(n^2,\langle4\rangle\perp \begin{pmatrix}4&2 \\ 2&2^33^t+1\end{pmatrix}\right).
 $$
 Similarly, we have $r(n^2,K_{2,t})=2r(n^2,\langle1,4,2^53^t\rangle)$.  Let $K_{3,t}=\z x_1+\z x_2+\z x_3$ such that 
$$  
(B(x_i,x_j))=\begin{pmatrix} 2&0&1 \\0&2&1 \\1&1&2^33^t+1\end{pmatrix}.
$$
Let $v \in K_{3,t}$ such that $Q(v)=n$.  Let $a,b,c$ be integers such that $v=a(x_1-x_2)+b(x_2+x_3)+cx_3$. Since  $Q(v) \equiv b^2+c^2 \equiv 1 \pmod 2$,
either $b$ or $c$ is odd, but not both. Hence for any odd integer $n$,
$$
r(n, K_{3,t})=r(n, \z(x_1-x_2)+\z(2x_2+2x_3)+\z x_3)+r(n,\z(x_1-x_2)+\z(x_2+x_3)+\z(2x_3)).
$$
Since 
$$   
 \z(x_1-x_2)+\z(2x_2+2x_3)+\z x_3 \simeq \z(x_1-x_2)+\z(x_2+x_3)+\z(2x_3) \simeq \langle4\rangle\perp \begin{pmatrix}4&2 \\ 2&2^33^t+1\end{pmatrix},
$$
we have
$$
r(n^2,K_{3,t})=2r\left(n^2,\langle4\rangle\perp \begin{pmatrix}4&2 \\ 2&2^33^t+1\end{pmatrix}\right),
$$
for any odd integer $n$. Consequently,  for any integer $n$,
$$
2r(n^2,K_{1,t})=r(n^2,K_{2,t})+r(n^2,K_{3,t}).
$$ 
This completes the proof. 
\end{proof}

\begin{prop} \label {main} Both ternary $\z$-lattices $K_{1,0}$ and $K_{1,1}$ defined above are strongly $s$-regular  ternary $\z$-lattices.  
\end{prop}
\begin{proof}  Note that 
$$
\gen(K_{1,0})=\{K_{1,0}, K_{2,0}, K_{3,0}\}\quad \text{and} \quad \gen(K_{1,1})=\{K_{1,1}, K_{2,1}, K_{3,1}\}.
$$ 
Therefore by Lemma \ref {classnumber3}, we have $$r(n^2,\gen(K_{1,i}))=4\left(\frac{1}{8}r(n^2,K_{1,i})+\frac{1}{16}r(n^2,K_{2,i})+\frac{1}{16}r(n^2,K_{3,i})\right)=r(n^2,K_{1,i}),$$for any integer $n$ and any $i=0, 1$. Therefore by Lemma 2.3 of \cite {ko}, we see that both $K_{1,0}$ and $K_{1,1}$ are  strongly $s$-regular.   
\end{proof}

\begin{rmk}
If a strongly $s$-regular  lattice $M$ has class number two, then the other lattice in the genus of $M$ is also strongly $s$-regular by Remark 3.3 of \cite{ko}.   This is not true in general if the class number of a lattice is greater than two. For example, both ternary $\z$-lattices $K_{1,0}$ and $K_{1,1}$ are strongly $s$-regular,  however all the other lattices in  $\gen(K_{1,0})$ and $\gen(K_{1,1})$ are not strongly $s$-regular. 
\end{rmk}

For any positive integer $t$, we define
$$
\ell_t=\begin{pmatrix} 1&\frac12 \\ \frac12&1\end{pmatrix} \perp \langle 3t\rangle, \ \ L_t=\begin{pmatrix} 1&\frac12 \\ \frac12&7\end{pmatrix} \perp \langle 3t\rangle,\ \ M_t=\begin{pmatrix} 1&\frac12&\frac12 \\ \frac12&7&\frac52 \\ \frac12&\frac52&3t+1\end{pmatrix}
$$
and 
$$
N_t=\begin{pmatrix} 3&\frac32&0 \\ \frac32&3&\frac32 \\ 0&\frac32&3t+1\end{pmatrix}, \ \  K_t=\begin{pmatrix} 1&\frac12 \\ \frac12&1\end{pmatrix} \perp \langle 27t\rangle.
$$

\begin{lem} \label{nonclassic3} Let $t$ be any positive integer. For any positive integer $n$, we have
$$
r(3n+1,\ell_t)=3r(3n+1,L_t)=3r(3n+1,M_t)=2r(3n+1,N_t)+r(3n+1,K_t).
$$
\end{lem}

\begin{proof} Let $\{x_1,x_2,x_3\}$ be the basis for $\ell_t$ whose Gram matrix is given above. Assume that $Q(ax_1+bx_2+cx_3)=3n+1$. Since $(a-b)^2\equiv 1 \pmod 3$, we have $a\equiv 0 \pmod 3$ or $b\equiv 0 \pmod 3$ or $a+b\equiv 0 \pmod 3$, however any of two cases cannot occur simultaneously. Therefore we have
$$
\begin{array} {rl}
r(3n+1,\ell_t)&\!\!\!\!\!=r(3n+1,\z(3x_1)+\z x_2+\z x_3)+r(3n+1,\z x_1+\z (3x_2)+\z x_3)\\
              &+r(3n+1,\z(x_1-x_2)+\z (3x_2)+\z x_3),
 \end{array}              
$$
which implies that $r(3n+1,\ell_t)=3r(3n+1,L_t)$. Now let $y_1=x_1,y_2=x_2$ and $y_3=x_1+x_2+x_3$. Then 
$$
(B(y_i,y_j))=\begin{pmatrix} 1&\frac12&\frac32 \\ \frac12&1&\frac32 \\ \frac32&\frac32&3t+3\end{pmatrix}.
$$
Assume $Q(ay_1+by_2+c y_3)=3n+1$. Since $(a-b)^2\equiv 1 \pmod 3$, we have $a\equiv 0 \pmod 3$ or $b\equiv 0 \pmod 3$ or $a+b\equiv 0 \pmod 3$, however any of two cases cannot occur simultaneously. Therefore we have
$$
\begin{array} {rl}
r(3n+1,\ell_t)&=r(3n+1,\z(3x_1)+\z x_2+\z x_3)+r(3n+1,\z x_1+\z (3x_2)+\z x_3)\\
              &+r(3n+1,\z(x_1-x_2)+\z (3x_2)+\z x_3),
 \end{array}              
$$
which implies that $r(3n+1,\ell_t)=3r(3n+1,M_t)$. Finally, if we choose a basis $\{x_1,x_1+x_2,x_1+x_3\}$ for $\ell_t$, then we may prove that 
 $r(3n+1,\ell_t)=2r(3n+1,N_t)+r(3n+1,K_t)$.
\end{proof}

\begin{prop} \label{classnumber3-1} The ternary $\z$-lattices $L_1,L_4,L_{10}$ and $M_6$ are all strongly s-regular.  
\end{prop}

\begin{proof} First, note that $N_t$ is contained in $\gen(K_t)$ for any positive integer $t$, $L_t\in  \gen(N_t)$ if $t \equiv 1 \pmod 3$, and $M_t \in  \gen(N_t)$ if $t \equiv 0 \pmod 3$. For any integer $t \not \equiv 0 \pmod 3$, since $\lambda_3^2(L_t)\simeq \lambda_3^2(N_t)\simeq \lambda_3^2(K_t)\simeq \ell_t$, 
we have 
\begin{equation} \label{rr}
r(9n,L_t)=r(9n,N_t)=r(9n,K_t)=r(n,\ell_t).
\end{equation} 
 If $t \equiv 0 \pmod 3$, then we have 
$$
r(9n,M_t)=r(9n,N_t)=r(9n,K_t)=r(n,\ell_t). 
$$ 
For $t=1,4$ or $10$, one may easily show that $\gen(L_t)=\{L_t,N_t,K_t\}$ and $\gen(M_6)=\{M_6,N_6,K_6\}$. Now 
by equation \eqref{rr} and Lemma \ref{nonclassic3}, we have
$$
r(n^2,L_t)=4\left(\frac{r(n^2,L_t)}8+\frac{r(n^2,N_t)}{12}+\frac{r(n^2,K_t)}{24}\right)=r(n^2,\gen(L_t)),
$$
for any $t=1,4$ or $10$. Furthermore, we have  
$$
r(n^2,M_6)=\frac83\left(\frac{r(n^2,M_6)}4+\frac{r(n^2,N_6)}{12}+\frac{r(n^2,K_6)}{24}\right)=r(n^2,\gen(M_6)).
$$
This completes the proof.   \end{proof}

\begin{thm}  Let $L$ be a ternary $\z$-lattice representing $1$. Then $L$ is strongly $s$-regular if and only if $L$ satisfies $r(n^2,L)=r(n^2,\gen(L))$ for any integer $n$.    
\end{thm}

\begin{proof} Note that ``if" is trivial. The  ``only if" is the direct consequence of Theorem \ref  {important-3} and Propositions \ref{classnumber2}, \ref{main} and \ref{classnumber3-1}.
\end{proof}


\begin{thebibliography}{abcd}


\bibitem {co} W. K. Chan and B.-K. Oh, {\em Finiteness theorems for positive definite $n$-regular quadratic forms}, Trans. Amer. Math. Soc. \textbf {355}(2003), 2385-2396.

\bibitem {cl} S. Cooper and H.-Y. Lam, {\em On the diophantine equation $n^2=x^2+by^2+cz^2$}, J. Number Theory \textbf {133}(2013), 719-737.

\bibitem {hjs} J. S. Hsia, M. J\"ochner, and Y. Y. Shao, {\em A structure theorem for a pair of quadratic forms}, Proc. Amer. Math. Soc. \textbf {199}(1993), 731-734.

\bibitem {jlo} J. Ju, I. Lee and B.-K. Oh, {\em A generalization of Watson transformation and representations of ternary quadratic forms}, submitted.



\bibitem {ko} K. Kim and B.-K. Oh, {\em The number of representations of squares by integral ternary quadratic forms}, submitted.


\bibitem {ki} Y. Kitaoka, {\em Arithmetic of quadratic forms}, Cambridge University Press, 1993.


\bibitem {om} O. T. O'Meara, {\em Introduction to quadratic forms}, Springer Verlag, New York, 1963.

\end{thebibliography}
\end{document}